\documentclass{amsart}

\usepackage{amsmath} 
\usepackage{amsfonts} 
\usepackage{amssymb} 
\usepackage{epsfig} 
\usepackage{graphicx} 
\usepackage{multirow} 
\usepackage{verbatim} 
\usepackage{rotating} 
\usepackage[all]{xy}

\newtheorem{theorem}{Theorem}[section]
\newtheorem{lemma}[theorem]{Lemma}

\theoremstyle{definition}
\newtheorem{definition}[theorem]{Definition}

\theoremstyle{remark}
\newtheorem{remark}[theorem]{Remark}
\theoremstyle{notation}
\newtheorem{notation}[theorem]{Notation}
\numberwithin{equation}{section}
\theoremstyle{corollary}
\newtheorem{corollary}[theorem]{Corollary}

\usepackage[bookmarks=false]{hyperref}
\newcommand{\Map}{\mathbf{Map}}
\newcommand{\map}{\mathbf{map}}

\newcommand{\sSet}{\mathbf{sSet}}

\newcommand{\Cat}{\mathbf{Cat}}

\newcommand{\Top}{\mathbf{Top}}

\newcommand{\C}{\mathbf{C}}

\newcommand{\M}{\mathbf{M}}

\newcommand{\Path}{\mathcal{P}}

\newcommand{\V}{\mathbf{V}}

\newcommand{\N}{\mathrm{N}_{\bullet}}

\newcommand{\Ho}{\mathrm{Ho}}

\begin{document}

\title{Homotopy Theory of  $T$-Algebras over $\Cat_{\Top}$ ?}

\author{Amrani Ilias}
\address{Department of Mathematics\\ Masaryk University\\ Czech Republic}
\email{ilias.amranifedotov@gmail.com}
\email{amrani@math.muni.cz}
\thanks{Supported by the project CZ.1.07/2.3.00/20.0003
of the Operational Programme Education for Competitiveness of the Ministry
of Education, Youth and Sports of the Czech Republic.
}

\thanks{}

\subjclass[2000]{Primary 18, 55}




\begin{abstract}
In this article, we interconnect two different aspects of higher category theory, in one hand the theory of $\infty$-categories and on an other hand the theory of $2$-categories.We construct an explicit functorial path objet in the model category of topological categories. We discuss some properties and consequences of such path object. We also explain the construction of a 2-monad which algebras are (symmetric) monoidal topological categories. Finally, we explain the relationship with the eventual model structure 
on the category of $T$-algebras.   
\end{abstract}

\maketitle

\section*{Motivation}
This work was motivated by the following question: \textsf{is there an induced model structure on the category of (symmetric) monoidal topological (simplicial) categories, or more generally a model structure on the category of $T$-algebras, where $T$ is some flexible (i.e., cofibrant) 2-monad \cite{lack2007}}? Trying to answer this question, we end up with a very simple (but stronger) question: \textsf{Is there a functorial path object in the model  category of topological categories having the property to be cartesian i.e.,  commutes with the cartesian product}? The existence of such functorial categorical path object and a functorial cartesian factorization of the diagonal implies the existence of an induced model structure on the category of  $T$-algebras  for some flexible 2-monads i.e., 2-operads \ref{2-operad}. This point of view is maybe naive, but as experience shows, these questions are the key point of the homotopy theory of $\mathcal{O}$-algebras, where $\mathcal{O}$ is a symmetric (cofibrant) operad over the category of spaces or simplicial sets  \cite{berger2003axiomatic}. As far as I understand, Lurie's books are motived by two essential aspects, the first one is homotopy theory of quasi-categories \cite{lurie2009higher} and the second one is the homotopy theory of quasi-categoires with some algebraic stucture (e.g. $\infty$-symmetric monoidal categories) codified by a notion of $\infty$-operads \cite{lurie2011higher}. Our contribution is very modest, we explain the construction of some particular 2-operads. Then we construct a functorial path objects in the category of topological categories but which fails to be a cartesian functor.
We carefully treated all the details. So the previous questions are still open. In section \ref{monad}, we introduce the theory of 2-operads. In section \ref{moorepath}, we construct an explicit path object in the category $\Cat_{\Top}$, which gives a new easy proof of the fact that $\Cat_{\Top}$ is a cofibrantly generated model category Quillen equivalent to $\Cat_{\sSet}$. In the last section \ref{eventual}, we discuss some properties and consequences of our functorial categorical path object. We end with some observation about the existence of an eventual model structure on $T$-algebras over the category of topological categories.

   \textbf{Acknowledgment: }
 I would like to thank Prof. Jiri Rosicky to pointing out the existence of the category $\Top_{\Delta}$ and John Bourke for his explanations about the 2-categorical world. I'm grateful to Prof. Steve Lack for all his precisions about  the theory of 2-monads. I'm also grateful to Alexandru Stanculescu and Lukas Vokrinek for their comments during the category theory seminar at Masaryk University. Finally, I would like to thank J. Lurie for answering my naive questions. 


\section{2-Monads, 2-Operads}\label{monad}
This part is a brief introduction to the theory of 2-monads and their algebras in the setting of 2-categories. The main references for 2-monads are \cite{kelly1993}, \cite{lack2007} and \cite{lack2010}. The Questions in mind is the following: \textit{Can we describe the (symmetric) monoidal enriched categories as algebras over a 2-monad?} It is quite clear that such algebraic structure can't be described with the theory of operads.  Let $\C$ be a locally presentable 2-category \cite{adamek1994locally} and $\C_{\lambda}$ the small 2-subcategory of $\lambda$-presentable objects. The category of $\lambda$-ranked endofunctors $\mathrm{End}_{\lambda}(\C)$  form a locally presentable monoidal 2-category. The category of monoids in $\mathrm{End}_{\lambda}(\C)$ is denoted by $\mathrm{Mnd}_{\lambda}(\C)$, which is the
2-category of $\lambda$-ranked 2-monads. In \cite{monadicity}, Lack proved that $\mathrm{Mnd}_{\lambda}(\C)$ is actually monadic i.e., we have an adjunction:
$$  \xymatrix{\mathrm{End}_{\lambda}(\C) \ar@<2pt>[r]^{ \mathbb{H}} & \mathrm{Mnd}_{\lambda}(\C)  \ar@<2pt>[l]^{forget}}$$ 
and the objects of the category $\mathrm{Mnd}_{\lambda}(\C)$ are precisely algebras over the monad generated by the previous adjunction. The functor $\mathbb{H}$ is the functor which associate to each $\lambda$-ranked endofunctor a  $\lambda$-ranked free monad. Recall that this result is far to be obvious because the category $\mathrm{End}_{\lambda}(\C)$ is monoidal but \textit{not} symmetric monoidal category.\\
The previous free functor allow us to construct a presentation for a given monads using weighted colimites (\textit{coequifers, co(iso)inserters}), for example
in the case of $\C=\Cat$, Lack \cite{lack2007}, gave a general presentation of monads on $\Cat$ which algebras are monoidal categories.\\ 
Our main interest is the enriched case. The category $\V=\sSet$ or $\Top_{\Delta}$ is locally presentable \cite{fajstrup2008convenient}.  By \cite{kelly2001v} we have that $\Cat_{\V}$ is a locally presentable. Thus we get the 2-adjunction 
$$ 
 \xymatrix{\mathrm{End}_{\lambda}(\Cat_{\V}) \ar@<2pt>[r]^{ \mathbb{H}} & \mathrm{Mnd}_{\lambda}(\Cat_{\V})  \ar@<2pt>[l]^{forget}}$$
There exists a monad for which algebras are monoidal (symmetric) simplicial categories. This construction follows form the 2-categorical structure on $\mathrm{End}_{\lambda}(\Cat_{\V}) $ ($\Cat$-weighted colimits).\\

\begin{definition}\label{T-algebras}
Given a monad $(T,\mu,\eta)$ on a category $\C$, a $T$-algebra $X$ is an object in $\C$ with a map 
$h:TX\rightarrow X$ subject to the following commutative diagrams:
$$
 \xymatrix{
 T^{2} X\ar[d]^{\mu_{X}}\ar[r]^{Th} & TX\ar[d]_{h} & & X\ar[r]^{\eta_{X}}\ar[dr]^{id}&TX\ar[d]^{h}\\
 TX \ar[r]^{h} & X. & & & X.
 }
 $$

\end{definition}
\begin{definition}\label{morphismT-algebras}
A morphism of $T$-algebras $h: TX\rightarrow X$ and  $g: TY\rightarrow Y$
 is a map $f:X\rightarrow Y$  in $\C$ subject to the following commutative 
diagram:
$$
\xymatrix{
 TX\ar[d]^{Tf}\ar[r]^{h} &  X\ar[d]_{f} \\
 TY \ar[r]^{g} & Y. 
 }
 $$

\end{definition}

\subsection{Presentation of a 2-monad}
We will explain the construction of a 2-monad $S$ on $\Cat_{\V}$ such that the $S$-algebras are exactly the monoidal symmetric enriched $\V$-categories and morphisms are the strict monoidal enriched functors.\\

 \begin{notation}\label{exmonads}
We fixe the following notations for some particular 2-monads over the 2-category $\Cat_{\V}$.
\begin{enumerate}
\item The $M$-algebras are enriched  monoidal categories (without unit)  and morphisms are strict monoidal functors. We denote denote the category of $M$-algebras by  $\Cat_{\V}^{M}$.
\item The $S$-algebras are enriched symmetric monoidal categories (without unit) and morphisms are strict monoidal functors. We denote denote the category of $S$-algebras by  $\Cat_{\V}^{S}$.
\end{enumerate}
\end{notation}

Before building an explicit presentation for 2-monads we explain some mechanism developed in \cite{lack2007}. For our propose $\M$ is $\lambda$-locally presentable 2-category. For any two objects $A$ and $B$ ones define the corresponding $\lambda$-ranked endofunctor $<A,B> ~\in \mathrm{End}_{\lambda}(\M)$ which satisfies the following isomorphism of categories (i.e., enriched homs) : 
\begin{equation}\label{adjunction1} 
\underline{\mathrm{End}_{\lambda}(\M)}(T,<A,B>)\simeq \underline{\M}(TA,B).
\end{equation}
The endofunctor $<A,A>$ has a structure of a monad. If $T$ is a monad, then a map of monads 
$T\rightarrow <A,A>$ corresponds a $T$-algebra structure on $A$ i.e., a map $TA\rightarrow A$ satisfying the commutative diagram defined in \ref{T-algebras}. In order to get a little intuition how to present a 2-monad, we start with an example. Suppose that $U^{n}: \M\rightarrow \M$ be a 2-functor which sends 
$A$ to $A^{\times^{n}}$, where $\M=\Cat$ or $ \Cat_{\V}$. A nature question arises: \textit{what are the algebras over the free monad $\mathbb{H}U^{n}$?} By definition an $\mathbb{H}U^{n}$-algebra structure on $A$ is a map
of monads $\mathbb{H}U^{n}\rightarrow<A,A>$ which is the same thing as giving a map $U^{n}\rightarrow <A,A>$ in $\mathrm{End}_{\lambda}(\M)$, and by the isomorphism given in \ref{adjunction1}, it is the same thing as defining a morphism $U^{n}A\rightarrow A$ in $\M$. We conclude that a $\mathbb{H}U^{n}$-algebra $C$ is the same thing as giving a n-array operation in $C$ i.e., a map $f: ~C^{\times^{n}}\rightarrow C$.

 \begin{remark}\label{representablefunctor1}
 Suppose we have two 2-monads $L$ and $R$ on $\Cat_{\V}$, then we have a one to one bijection between diagrams of the form:
  $$
  \xymatrix{
 \Cat_{\V}^{R} = R-\mathbf{Alg}\ar[rd]_-{forget}\ar[rr]^-{f} & & L-\mathbf{Alg}=\Cat_{\V}^{L} \ar[ld]^-{forget} \\
  &\Cat_{\V}.  
  }
  $$
 and morphisms of 2-monads $\mathrm{Mnd}_{f}(\Cat_{\V})(L,R)$.
 
 \end{remark}
 
Consider the free monads $\mathbb{H}U^{2}$ and $\mathbb{H}U^{3}$ generated by $U^{2}$ and $U^{3}$. Take the co-isoinsertrer in $\mathrm{Mnd}_{f}(\Cat_{\V})$ of the two maps $f, g: \mathbb{H}U^{3}\rightarrow \mathbb{H}U^{2}$ which are defined as follow: each $\mathbb{H}U^{2}$-algebra has two  natural $\mathbb{H}U^{3}$-algebra structures \ref{representablefunctor1}. The co-isoinsertrer of the maps $f$ and $g$ is a map $r: \mathbb{H}U^{2}\rightarrow C$ such that we have a monad isomorphism $\rho: rf\simeq rg.$ An $C-$algebra $\C$ is now an enriched category with a functor $\otimes: \C\times\C\rightarrow \C,$ and a natural isomorphism $\alpha : (a\otimes b)\otimes c\simeq a\otimes (b\otimes c).$

Now we consider the 2-functor $U^{4}$. There are two $\mathbb{H}U^{4}-$algebra structures on a $C$-algebra, related to the derived operations 
$((a\otimes b)\otimes c)\otimes d$ and $a\otimes ((b\otimes c)\otimes d)$ which induce two monads maps
$f', g': \mathbb{H}U^{4}\rightarrow C$ \ref{representablefunctor1}. The two isomorphisms $((a\otimes b)\otimes c)\otimes d\simeq a\otimes ((b\otimes c)\otimes d)$ which can be build via $\alpha$ induce two monad transformations $\phi, \psi: f^{'} \rightarrow g'$, and we can now form the coequifier of the two cells, namely the universal monad map $q:C\rightarrow M$ for which $q\psi= q\phi$. The $M-$algebras are exactly the monoidal enriched categories without unit object. For each $M-$algebra there is an opposite $M$-algebra structure. Now, we will build a 2-monad which algebras are exactly \textit{symmetric} monoidal enriched categories. Each monoidal enriched category $(\C, \otimes)$ has two monoidal structures, the original one and the opposite one $(\C, \otimes^{op})$. Thus, we have two maps of monads $i,j: ~\mathbb{H}U^{2}\rightarrow M$. We form a co-isoinserer $c:M\rightarrow S^{'}$. The $S^{'}$-algebras are monoidal categories such that we have a natural isomorphism $\sigma:~ a\otimes b\rightarrow b\otimes a$.  
For the symmetric monoidal structure, one of the axioms requires that $\sigma ^2=id$. So again any $S^{'}-$algebra $(\C, \otimes)$ has two $S^{'}$-algebra structures the original one and the one $(\C, \otimes^{\sigma^{2}})$ given by  symmetric structure $a\otimes^{\sigma^{2}}  b= \sigma^{2}(a\otimes b)$ i.e., twisting two times. It induce to morphism of monads
 $h,d: ~\mathbb{H}U^{2}\rightarrow S^{'}$. We form the coequifier   $p:S^{'}\rightarrow S^{''}$. An $S^{''}$-algebra $\C$ is a monoidal topological category with a natural twist $\sigma: a\otimes b\rightarrow b\otimes a$ such that $\sigma^{2}= id$.
 Finally using again the coequifier for $x,y:~\mathbb{H}U^{3}\rightarrow S^{''}$ which take in account the compatibility between the symmetry and the associative laws i.e., the following commutative diagram:
 $$
  \xymatrix{
   (a\otimes b) \otimes c\ar[d]\ar[r] & a\otimes (b \otimes c)\ar[r] & (b\otimes c) \otimes a\ar[d] \\
  (b\otimes a) \otimes c\ar[r] & b\otimes (a \otimes c)\ar[r] & b\otimes (c \otimes a).  
  }
  $$
  we end-up with the monad $S$ which algebras are exactly symmetric monoidal enriched categories (without unit object). 
 
 \begin{remark}
 The 2-monad which algebras are (symmetric) monoidal enriched unital categories is build by the same way by taking in account the operation of type 0. 
 \end{remark} 
We should remark that we have introduced a very special 2-monads which codify linear operations as the theory of operads in the classical sense \cite{berger2003axiomatic}. These 2-monads are build using only the free 2-monads of the form $\mathbb{H}U^{n}$.   
\begin{definition}\label{2-operad}
A 2-monad $T$ is called a 2-operad if it is build from free monads $\mathbb{H}U^{n}$ using  coproducts, coinserters and coequifiers.
\end{definition}

\section{Categorical Path Object}\label{moorepath}
Let $I=[0,1]$, and $X$ a topological space. It is well known that in the standard path object $X^{I}=\map(I,X)$  the concatenation of paths (when it is possible) is \textit{not} strictly associative. To solve this inconvenient problem we introduce the \textit{Moore path space}.
\begin{definition}\label{moorepathsapce}
Let $X$ be a topological space, and $\mathcal{M}(X)$ be the subspace of $ X^{\mathbb{R}^{+}}\times\mathbb{R}^{+}$ with the following property: 
$$\mathcal{M}(X)=\{(\alpha:\mathbb{R}^{+}\rightarrow X, r)\in X^{\mathbb{R}^{+}}\times\mathbb{R}^{+} | ~\alpha (s)=\alpha (r), ~\forall s\geqslant r\}.$$ 
The real number $r$ is called the length of the path.  
\end{definition}
There is a natural inclusion of $X^{I}=\mathcal{M}^{1}(X)\rightarrow \mathcal{M}(X)$ wich is the inclusion of path of length 1.
This inclusion is a weak equivalence. Moreover, we have two natural maps 
$(\delta_{-},\delta_{+}):  \mathcal{M}(X)\rightarrow X$ defined by $\delta_{-}(\alpha, r)= \alpha(0)$
and   $\delta_{+}(\alpha, r)= \alpha(r)$.
\begin{lemma}
The map $(\delta_{-},\delta_{+}):~ \mathcal{M}(X)\rightarrow X\times X$ is a fibration.
\end{lemma}
\begin{proof}
Let $I^{n}\rightarrow I^{n+1}$ be a generating trivial cofibration, and consider a commutative diagram: 
 $$
 \xymatrix{
 I^{n}\ar[d]^{}\ar[r]^-{f} & \mathcal{M}(X)\ar[d]^-{(\delta_{-},\delta_{+})}\\
I^{n+1} \ar[r] & X\times X.
 }
 $$ 
 Since $I^{n}$ is compact, there is some $r\in\mathbb{R}^{+}$ such that the image of $f$ lives in $\mathcal{M}^{r}(X)\subset\mathcal{M}(X) $. Consequently, we have to construct a lifting for the follwing diagram:
  
  $$
 \xymatrix{
 I^{n}\ar[d]^{}\ar[r]^-{f} & \mathcal{M}^{r}(X)\subset \mathcal{M}(X)\ar[d]^-{(ev_{0},ev_{r})}\\
I^{n+1} \ar[r]\ar@{.>}[ru] & X\times X.
 }
 $$ 
where $ev_{0}$ and $ev_{r}$ are evaluation maps at 0 and $r$ respectively. But the existence of a lift is obvious, since $\mathcal{M}^{r}(X)$ is just the rescaling of $X^{I}$ by $r$.
\end{proof}
\begin{remark}
The diagonal map $\Delta: ~X\rightarrow X\times X $ is factored by $X\rightarrow\mathcal{M}(X)\rightarrow X\times X $, where the first map (weak equivalence) is the evident embedding of $X$ in $ \mathcal{M}(X)$ as constant paths of length 0. The second map is the fibration $(\delta_{-},\delta_{+})$.
\end{remark}
\begin{definition}
Let $(\alpha, r)$ and $(\beta, s)$ two elements of $\mathcal{M}(X)$ such that 
$\alpha(r)=\beta(0)$, then we can define a strict composition $(\beta\ast\alpha, r+s) $ by concatenation. 
\end{definition}
\begin{remark}
The law composition $\ast$ is strictly associative and unital when it is defined. 
\end{remark}

\subsection{Categorical Path Object}\label{section2}
In order to construct a functorial path object in the category  $\Cat_{\Top}$  we took our inspiration from  \cite{drinfeld2004dg} and \cite{tabuada2007new}. For the model structure on the category of topological categories we refer to \cite{Amrani1}.

 \begin{definition}\label{pathobject}
Let $\Path: \Cat_{\Top}\rightarrow \Cat_{\Top}$ be defined as follow:
Given a topological category $\M$, we declare:
\begin{enumerate}
\item Objects of $\Path\C$ are morphisms $\xymatrix{a: x\ar[r]^-{\sim}& y}$ such that $a$ is a weak homotopy equivalence in $\C$. 
\item The topological set of maps $\Map_{\Path\C}(x\rightarrow y,x^{'}\rightarrow y^{'})$ between two objects $a:x\rightarrow y$ and $b:x^{'}\rightarrow y^{'}$  is given by the pullback of the diagram in $\Top$. 

    $$
  \xymatrix{
  \Map_{\Path\C}(x\rightarrow y,x^{'}\rightarrow y^{'})\ar[d]_-{(ev_{s},ev_{t})}\ar[r]  &   \mathcal{M}(\Map_{\C}(x,y^{'}))\ar[d]^{(\delta_{-},~\delta_{+})} \\
  \Map_{\C}(x,x^{'})\times \Map_{\C}(y,y^{'})  \ar[r]_{a^{\ast},b_{\ast}}&  \Map_{\C}(x,y^{'})\times \Map_{\C}(x,y^{'})
  }
  $$
  where $\mathcal{M}(\Map_{\C}(x,y^{'}))$ is the Moore path space associated to $\Map_{\C}(x,y^{'})$ \ref{moorepathsapce}.
  \end{enumerate}
\end{definition}
\begin{remark}
 To understand maps of the category of $\Path\C$ we should say that $$\Map_{\Path\C}(x\rightarrow y,x^{'}\rightarrow y^{'})$$ is a model for the homotopy pullback in $\Top$ of the diagram 
 
  $$
  \xymatrix{
    &  \Map_{\C}(y,y^{'})\ar[d]^{a^{\ast}} \\
  \Map_{\C}(x,x^{'}) \ar[r]_{b_{\ast}}&  \Map_{\C}(x,y^{'}).
  }
  $$
\end{remark}


  \begin{remark}
 The diagonal map $\Delta:\C\rightarrow \C\times\C$ is factored as 
 $$
  \xymatrix{
   \C\ar[rr]^{\Delta}\ar[rd]^{i} & & \C\times\C \\
  &\Path\C . \ar[ru]_-{(ev_{s},ev_{t})}& 
  }
  $$
  Where
  \begin{itemize}
  \item  $i:\C\rightarrow\Path\C$ is defined as $i(x)= id:~x\rightarrow x$.
  \item  $ev_{s}(x\rightarrow y)=x$ and  $ev_{t}(x\rightarrow y)=y$.
  \end{itemize}
 \end{remark}
 \begin{lemma}\label{lemma1}
 The map $i:\C\rightarrow\Path\C$ is a Dwyer-Kan weak equivalence.
 \end{lemma}
 \begin{proof}
 First, let us check that $i$ is homotopicaly fully faithful i.e.,
 $$\Map_{\C}(x,y)\rightarrow\Map_{\Path\C}(i(x),i(y))$$
 is a weak equivalence of topological space, since by definition 
 $$ \Map_{\Path\C}(i(x),i(y))=\mathcal{M}(\Map_{\C}(x,y)).$$
 For the homotopical essential surjectivity, lets take an object $a: x\rightarrow y$ 
 but, by definition, there is a weak equivalence in $\Path\C$ between $i(a)$ and 
 $a:x\rightarrow y$ given by
 $$
 \xymatrix{
 x\ar[d]^{id}\ar[r]^{id} & x\ar[d]^{\sim}_{a}\\
 x \ar[r]_{\sim}^{a} & y.
 }
 $$
so $i$ is an equivalence of topological categories.
 
 \end{proof}
\begin{lemma}\label{lemma2}
The map $(ev_{s},ev_{t}):\Path\C\rightarrow \C$ is a fibration in $\Cat_{\Top}$
\end{lemma}
\begin{proof}
First, we have to prove that 
$$(ev_{s}, ev_{t}):\Map_{\Path\C}(x\rightarrow y, x^{'}\rightarrow y^{'})\rightarrow \Map_{\C}(x,x^{'})\times \Map_{\M}(y,y^{'})$$ 
is a fibration of topological spaces. By definition $(ev_{s}, ev_{t})$ is a pullback of the fibration 
$$(\delta_{-},\delta_{+}):  \mathcal{M}(\Map_{\C}(x,y^{'}))\rightarrow \Map_{\C}(x,y^{'})\times \Map_{\C}(x,y^{'}),$$ and $\Top$ is right proper.   

Second, lets $w_{1},w_{2}:(x,y)\rightarrow (x^{'},y^{'})$ be a weak equivalence in $\C\times\C$ such
 that there exist an object $a:x^{'}\rightarrow y^{'}$ in $\Path\C$. Choose a homotopical inverse to $w_{2}$  
 denoted by $w_{2}^{-1} $ i.e., we have a path between $id$ and $w_{2}w_{2}^{-1}$, and define the map $w_{2}^{-1} a w_{1}: x\rightarrow y$ which is also a 
 weak homotopy equivalence. 
  $$
 \xymatrix{
 x\ar[d]^{\sim}_{w_{1}}\ar[r]^{w_{2}^{-1} a w_{1}} & y\ar[d]_{\sim}^{w_{2}}\\
 x^{'} \ar[r]^{\sim}_{a} & y^{'}.
 }
 $$
 
 We remark, by definition, that we have a path 
 $$\gamma: [0,1]\rightarrow \Map_{\C}(x,y^{'})$$
  between $w_{2} w_{2}^{-1} a w_{1} $ and $a w_{1}$. 
  So we can lift $(w_{1},w_{2})$ to a weak equivalence in $\Path\C$ given by $(w_{1},w_{2}, \gamma).$
\end{proof}
\begin{lemma}\label{lemma3}
The categorical path object defined in \ref{pathobject} is a path object in the model category $\Cat_{\Top}$ constructed in \cite{Amrani1}  
\end{lemma}
\begin{proof}
It is a direct consequence of \cite{Amrani1}, \ref{lemma1} and \ref{lemma2}.
\end{proof}



\subsection{composition law in a categorical path object}\label{composition}
Now, we define the composition law in the categorical path object introduced in \ref{pathobject}.
Let $\C$ be a topological category. The categorical path object $\mathcal{P}(\C)$
has as objects $a:x\rightarrow y$ where $a$ is a weak equivalence in $\C$.
  A morphism in $\mathcal{P}(\C)$ between  $a: x\rightarrow y$ and $b: x^{'}\rightarrow y^{'}$ is a diagram of the form 
  $$ \xymatrix{ 
    x \ar[r]^{a} \ar[d]_{f}  & y \ar[d]^{g} \ar@{=>}[ld]_\alpha\\
    x^{'} \ar[r]^{b} & y^{'} 
    }$$
  where $f,~g$ are elements of $\Map_{\C}(x,x^{'})$ and $\Map_{\C}(y,y^{'})$, and $\alpha_{r}$ an element of 
   $ \mathcal{M}(\Map_{\C}(x,y^{'}))$ such that $\alpha$ is a homotopy of length $r$ between  $ g.a$ and $b.f$.
 We check the composition law. Let $a: x\rightarrow y$,  $b: x^{'}\rightarrow y^{'}$ and  $c: x^{''}\rightarrow y^{''}$ tree objects of $\mathcal{P}(\C)$. Consider the diagram 
$$ \xymatrix{ 
    x \ar[r]^{a} \ar[d]_{f} & y \ar[d]^{g} \ar@{=>}[ld]_{\alpha_{r}}\\
    x^{'} \ar[r]^{b}\ar[d]_{f^{'}} & y^{'}\ar[d]^{g^{'}} \ar@{=>}[ld]_{\beta_{s}} \\
    x^{''} \ar[r] ^{c}  & y^{''} 
    }$$
   by composition we obtain 
   $$ \xymatrix{ 
    x \ar[rr]^{a} \ar[dd]_{f^{'}f}&& y \ar[dd]^{g^{'}g} \ar@{=>}[lldd]_{\alpha\circ\beta_{r+s}}  \\
    & & \\
    x^{''} \ar[rr]^{b} && y^{''} 
    }$$
Where $\beta\circ\alpha_{r+s}\in  \mathcal{M}(\Map_{\C}(x,y^{''}))$ is defined as the concatenation
$\beta .f \ast g^{'}.\alpha$. The identity map $1$ of $a: x\rightarrow y$ is given by $id:x\rightarrow x$ , $id:y\rightarrow y$ and a constant Moore path of length 0. 
\begin{lemma}
The previous composition law is associative. 
\end{lemma}
 \begin{proof}
 Let take the following diagram:
 
 $$ \xymatrix{ 
    x \ar[r]^{a} \ar[d]_{f}  & y \ar[d]^{g} \ar@{=>}[ld]_{\alpha_{r}}\\
    x^{'} \ar[r]^{b}\ar[d]_{f^{'}} & y^{'}\ar[d]^{g^{'}}  \ar@{=>}[ld]_{\beta_{s}}\\
    x^{''} \ar[r] ^{c}\ar[d]_{f^{''}} & y^{''}\ar[d]^{g^{''}}\ar@{=>}[ld]_{\gamma_{t}}  \\
    x^{'''} \ar[r] ^{d}  & y^{'''} 
    }$$
  The composition of morphisms and Moore paths is strict and we get:
  \begin{eqnarray}
  (\gamma\circ\beta)\circ \alpha &=& (\gamma f^{'}f\ast g^{''}.\beta)\circ \alpha\\
                                                &=& (\gamma . f^{'}\ast g^{''}.\beta).f \ast g^{''}g^{'}.\alpha\\
                                                &=& \gamma . f^{'}f\ast g^{''}.\beta .f \ast g^{''}g^{'}.\alpha
                                                \end{eqnarray}
                                                
  \begin{eqnarray}
  \gamma\circ(\beta\circ \alpha) &=& \gamma \circ (\beta . f\ast g^{'}.\alpha)\\
                                                &=& \gamma . f^{'}f \ast  g^{''}. (\beta . f\ast g^{'}. \alpha)\\
                                                &=& \gamma . f^{'}f\ast g^{''}. \beta . f \ast g^{''}g^{'}. \alpha
                                                \end{eqnarray}

 \end{proof}
                                               
 \begin{corollary}
 The category $\Cat_{\Top}$ is a cofibrantly generated model category and Quillen equivalent to the model category $\Cat_{\sSet}$.
\end{corollary}
We propose a new and very easy proof comparing with \cite{Amrani1}.
\begin{proof}
The standard Quillen adjunction between $\sSet$ and compactly generated spaces $\Top$ induces an adjunction between $\Cat_{\sSet}$ and $\Cat_{\Top}$. The existence of an induced model structure is a direct consequence of the existence of path object \ref{lemma3}, \cite{bergner} and [\cite{berger2003axiomatic}, 2.6]. The Quillen equivalence is an obvious fact. 

\end{proof}
\subsection{Algebraic structure on $\Path\C$ }
The categorical path object is not cartesian since the Moore path objet does not commute with the cartesian product. Let $(\C,\otimes)$ be a monoidal topological category, then $\otimes$ extends obviously to a monoidal structure on objects of $\Path(\C)$. We denote by $\odot$ the extension of $\otimes$ by the formula 
$$(x\rightarrow y)\odot (x^{'}\rightarrow y^{'}):=x\otimes x^{'}\rightarrow y\otimes y^{'}.$$ 
The naive extension $$-\odot -: \mathcal{M}(\map_{\C}(x,y'))\times \mathcal{M}(\map_{\C}(w,z')) \rightarrow  \mathcal{M}(\map_{\C}(x\otimes w,y^{'}\otimes z^{'}) )$$
  which sends any Moore path $\alpha_{r}$ of length $r$ and any Moore path $\beta_{s}$ of length $s$ to the Moore
  path $(\beta\odot \alpha)(t):=\beta(t)\otimes \alpha(t)$ of length $r+s$  does \textit{not} extend naturally to a bifunctor $-\odot -: \mathcal{P}(\C)\times\mathcal{P}(\C)\rightarrow\mathcal{P}(\C)$, since it does not verify strictly (but only up to homotopy) the following equality:
  $$(\beta_{s}\odot\beta^{'}_{z})\circ(\alpha_{r}\odot\alpha^{'}_{v}) = (\beta_{s}\circ\alpha_{r})\odot(\beta^{'}_{z}\circ\alpha^{'}_{v}).$$
  Actually, this equality holds only in the case where $r=s$ and $z=v$.


\section{Eventual model structure on $T$-algebras.}\label{eventual}
This section si purely conjectural. We should first mention that 2-subcategory $\Cat_{\Top_{\Delta}}$ of the 2-category $\Cat_{\Top}$ is a combinatorial model category with the same weak equivalences, fibrations and generating (trivial) cofibration (same proof as in \cite{Amrani1}). If $T$ is a 2-operad on  $\Cat_{\Top_{\Delta}}$, then the category of $T$-algebra $\Cat_{\Top_{\Delta}}^{T}$ is (co)complete and we have an adjunction 
 
$$ 
 \xymatrix{\Cat_{\Top_{\Delta}}\ar@<2pt>[r]^{ F} & \Cat_{\Top_{\Delta}}^{T}  \ar@<2pt>[l]^{U}}$$

 Suppose that there is a functorial path $\mathcal{P}$ object on $\Cat_{\Top_{\Delta}}$, which commutes with the cartesian product, such that the factorization of the diagonal $\C\rightarrow\mathcal{P}(\C)\rightarrow \C\times\C$ is cartesian, then the category  $\Cat_{\Top_{\Delta}}^{T}$ is a cofibrantly generated model category with the underlying model structure, i.e., $f: A\rightarrow B$ is a weak equivalence (fibration) in  
 $\Cat_{\Top_{\Delta}}^{T}$ if and only if the map $Uf$ is a weak equivalence (fibration) in  $\Cat_{\Top_{\Delta}}$. The proof is an easy consequence of [\cite{berger2003axiomatic}, 2.6]. 
 
 \begin{remark}
 The existence of a cartesian categorical path objet is a sufficient but not a necessary condition. 
 \end{remark}
 
An other possible way to define a cartesian categorical path objet will be to find a functorial path object $\mathcal{C}$ in the category of topological spaces satisfying the following two conditions:
\begin{enumerate}
\item $\mathcal{C} (X\times Y)\simeq \mathcal{C} (X)\times \mathcal{C} (Y)$ i.e., $\mathcal{C}$ is a cartesian functor.
\item The functorial concatenation map (when it is defined) $\mathcal{C} (X)\times_{X}\mathcal{C} (X)\rightarrow \mathcal{C} (X)$ is strictly associative with unit.
\end{enumerate}
Consider the categorical path object $\Path^{'}$ defined as $\Path$ \ref{pathobject}, where we replace the Moore path object $\mathcal{M}$ by $\mathcal{C}$. It follows that $\Path^{'}$ is a categorical cartesian path object.
We should remark that such path object $\mathcal{C}$ in the category of topological spaces does not exist. This remark is due to J. Lurie. Roughly speaking, suppose that $G$ is a topological group. Then $\mathrm{\Omega} (G)$  (subspace of $\mathcal{C} (G)$) of loops based on the unit  would have an associative multiplication coming from the functorial concatenation, and an other associative multiplication coming from the group structure on G. By functoriality these would be compatible, so the Eckmann-Hilton argument would imply that both multiplications were the same and that they were commutative.
But the multiplication on the loop space of a topological group G, usually cannot be made strictly commutative.\\
\subsection{Classical operads vs 2-operads}
 At this stage of our article, we are tempted to ask the following question: \textsf{What is the relation between the classical operads (defined on simplicial sets) and the notion of 2-operads \ref{2-operad}} ? First of all, let us clarify the question by an example. Roughly speaking an $\infty$-groupoids $\C$ is a topological category such that its homotopy category $\pi_{0}\C$ is an ordinary groupoid i.e., a category where all morphisms are isomorphisms. Suppose that there is an induced  model structure on the category of symmetric monoidal topological categories  $\Cat_{\Top}^{S}$, ($S$ is the 2-operad discribed in \ref{exmonads}) then there is a restricted model structure on the category $\infty-\mathbf{Grp}^{S}$ of symmetric monoidal $\infty$-groupoids . The direct consequence is that we have a right Quillen functor 
 $$\N:  \infty-\mathbf{Grp}^{S}\longrightarrow \mathcal{E}_{\infty}-spaces$$
 where $\N$ is the coherent nerve. 
 In this example, it will be interesting to understand the properties of the derived right Quillen functor 
 $$\mathrm{R}\N:\Ho(\infty-\mathbf{Grp}^{S})\longrightarrow \Ho(\mathcal{E}_{\infty}-spaces).$$

 \subsection{$\mathcal{A}_{\infty}$-Categories}
We think that one of the problems to find a \textit{good} categorical path object, in the model category $\Cat_{\Top}$, is due to the fact that the category of topological monoids is too rigid. But we are not claiming that an induced model structure on  $\Cat_{\Top}^{T}$, for a 2-operad $T$, does not exist. The idea is to consider the 2-category of  $\mathcal{A}_{\infty}$-categories. There is actually many definitions of the notion of  $\mathcal{A}_{\infty}$-category. One definition is developed in the context of differential graded algebras and dg-categories, by M. Konstevich.
Our interest is concentrated on the case of topological (simplicial) enrichment. In \cite{blumberg2008topological}, the authors prove that there is a monoidal symmetric closed model structure on the category of topological spaces (more precisely $\ast$-modules) denoted by $(\mathbf{U},\boxtimes, \ast)$ such that the (commutative) monoids are exactly ($\mathcal{E}_{\infty}$-spaces) $\mathcal{A}_{\infty}$-spaces. The idea is to consider the 2-category  $\Cat_{\mathbf{U}}$ of enriched categories over $ \mathbf{U}$ as a model for $\mathcal{A}_{\infty}$-categories.
A recent work \cite{bergermoerdijk}  generalize the Dwyer-Kan model structure on the category of enriched categories. We think that C. Berger and I. Moerdijk ideas could be crucial and more natural in our understanding of homotopy theory of the category  $\Cat_{\mathbf{U}}^{T}$ of $T$-algebras.

\bibliographystyle{plain} 
\bibliography{infinity}

\end{document}